\documentclass[12pt,reqno]{article}

\usepackage[usenames]{color}
\usepackage{amssymb}
\usepackage{amsmath}
\usepackage{amsthm}
\usepackage{amsfonts}
\usepackage{amscd}
\usepackage{graphicx}
\usepackage{xcolor}
\usepackage{float}

\usepackage[colorlinks=true,
linkcolor=webgreen,
filecolor=webbrown,
citecolor=webgreen]{hyperref}

\definecolor{webgreen}{rgb}{0,.5,0}
\definecolor{webbrown}{rgb}{.6,0,0}

\usepackage{fullpage}

\usepackage{graphics}
\usepackage{latexsym}
\usepackage{epsf}
\usepackage{breakurl}

\newcommand{\seqnum}[1]{\href{https://oeis.org/#1}{\rm \underline{#1}}}
\newcommand{\ZZ}{\mathbb{Z}}
\begin{document}

\theoremstyle{plain}
\newtheorem{theorem}{Theorem}
\newtheorem{corollary}[theorem]{Corollary}
\newtheorem{lemma}[theorem]{Lemma}
\newtheorem{proposition}[theorem]{Proposition}

\theoremstyle{definition}
\newtheorem{definition}[theorem]{Definition}
\newtheorem{example}[theorem]{Example}
\newtheorem{conjecture}[theorem]{Conjecture}
\newtheorem{problem}[theorem]{Problem}

\theoremstyle{remark}
\newtheorem{remark}[theorem]{Remark}

\begin{center}
\vskip 1cm{\LARGE\bf The Hofstadter Consecutive-Sum Sequence \\
\vskip .08in
Omits Infinitely Many Positive Integers}
\vskip 1cm
\large
Quanyu Tang\\
School of Mathematical Sciences\\
University of Science and Technology of China\\
Hefei 230026\\
P.\ R.\ China\\
\href{mailto:tangquanyu827@gmail.com}{\tt tangquanyu827@gmail.com} \\
\end{center}

\vskip .2 in
\begin{abstract}
We study Hofstadter's greedy self-generating sequence, which begins with 1 and 2 and repeatedly chooses the least integer larger than the preceding term that is a sum of at least two consecutive earlier terms. We prove that the excess of its $n$th term over its index is bounded below by a positive constant times $\log\log n$, up to an additive constant, and that the sequence satisfies a polynomial upper bound. In particular, the sequence omits infinitely many positive integers.
\end{abstract}

\section{Introduction}

Douglas Hofstadter asked for a description of the asymptotic behavior
of the following greedy sequence. The question is recorded by
Erd\H{o}s~\cite[p.\ 71]{Er77c}, Erd\H{o}s and
Graham~\cite[p.\ 83]{ErGr80}, and Guy~\cite[E31]{Gu04}; Bloom also
lists it as Problem~\#423 on the \emph{Erd\H{o}s Problems}
website~\cite{EP}. In the classical case one begins with $(1,2)$ and obtains a strictly increasing sequence $(a_n)_{n\geq 1}$ as follows.

\begin{definition}\label{def:hof}
Set $a_1=1$ and $a_2=2$. For every $k\geq 3$, define $a_k$ to be the
least integer greater than $a_{k-1}$ that is a sum of at least two
consecutive earlier terms.
\end{definition}

In particular, Definition~\ref{def:hof} implies that, for every
$k\geq 3$,
\begin{equation}\label{eq:hof}
a_k=\sum_{i=p}^{q}a_i
\qquad\text{for some }1\leq p<q\leq k-1.
\end{equation}
Thus \eqref{eq:hof} records a property implied by the definition,
rather than the definition itself. This construction gives sequence
\seqnum{A005243} in the OEIS~\cite{OEIS}. Erd\H{o}s reports that
Hofstadter's question was inspired by a problem of Ulam about gaps
among consecutive sums of a finite increasing
sequence~\cite[p.\ 71]{Er77c}.

\begin{problem}\label{prob:423}
What is the asymptotic behavior of the sequence in
Definition~\ref{def:hof}?
\end{problem}

Numerically, the sequence appears to contain ``most'' integers. For orientation, the sequence in Definition~\ref{def:hof} begins
\[
1,2,3,5,6,8,10,11,14,16,17,18,19,21,22,24,25,29,30,32,\ldots,
\]
and hence it already omits $4,7,9,12,13,15,20,\ldots$. At the same time, the \emph{Comments} section of \seqnum{A005243} records the following conjecture~\cite{OEIS}.
\begin{conjecture}\label{conj:missing}
The sequence in Definition~\ref{def:hof} omits infinitely many positive integers.
\end{conjecture}

For later use, write
\[
b_n:=a_n-n\qquad(n\geq1).
\]
Our first result settles Conjecture~\ref{conj:missing}.
\begin{theorem}\label{thm:qualitative}
The sequence \((b_n)_{n\geq1}\) is nondecreasing and unbounded. Consequently,
\[
a_n=n+\omega(1),
\]
and the set \(\{a_n:n\geq1\}\) omits infinitely many positive integers.
\end{theorem}

In fact, building on Theorem~\ref{thm:qualitative}, we prove the following quantitative lower bound.
\begin{theorem}\label{thm:quantitative-lower}
We have
\[
a_n\geq n+\frac{\log\log n}{\log 20}-O(1).
\]
\end{theorem}

We also obtain a quantitative upper bound.
\begin{theorem}\label{thm:upper-final}
For every $\varepsilon>0$ there exists $C_\varepsilon>0$ such that, for all $n\geq1$,
\[
a_n\leq C_\varepsilon\,n^{4175/2506+\varepsilon}.
\]
\end{theorem}

The proof of Theorem~\ref{thm:upper-final} combines the greedy structure of the sequence with a recent lower bound for difference sets of finite convex sets. Thus, besides settling Conjecture~\ref{conj:missing}, we also obtain a first polynomial upper bound toward Problem~\ref{prob:423}. For related work on consecutive sums in integer sequences, see Beker~\cite{Beker2024}, Hegyv\'ari~\cite{Hegyvari1986}, and Tao~\cite{Tao2026}.

\subsection{Overview of the proofs}\label{subsec:proof-overview}

The proof of Theorem~\ref{thm:qualitative} begins with the observation that $b_n=a_n-n$ is nondecreasing. If $(b_n)$ were bounded, then it would eventually be constant, say $b_n=B$ with $B\geq0$, so $a_n=n+B$ on a tail. For every sufficiently large $r$, set $n=2^r-B$; then $a_n=2^r$, and this term must be a sum of consecutive earlier terms. For large $r$, such a representation is too large to lie entirely in the fixed initial prefix, while it cannot lie entirely in the linear tail because a power of $2$ is not a sum of at least two consecutive positive integers. It must therefore cross the boundary of the linear tail, and summing its linear part produces an equation
\[
v^2+v+E=2^k
\]
with $E$ taking values in a finite set. By the pigeonhole principle, one value of $E$ would then occur for infinitely many values of $r$, contradicting the Schinzel--Tijdeman finiteness theorem~\cite[Corollary~1]{SchinzelTijdeman1976}. Since $b_n=a_n-n$ counts the omitted positive integers up to $a_n$, this also gives the final assertion of the theorem.

For Theorem~\ref{thm:quantitative-lower}, we quantify how long $(b_n)$
can remain constant at a given integer value. By
Theorem~\ref{thm:qualitative}, $(b_n)$ is nondecreasing and unbounded;
since it is integer-valued, we have $b_n\to\infty$. For each integer
$B\geq0$, let $N(B)$ be the last index at which $b_n\leq B$, and set
\[
M(B)=N(B)+B+2.
\]
These quantities are well defined because $b_1=0$ and $b_n\to\infty$.

Now fix an integer $B\geq1$. Since $(b_n)$ is nondecreasing and
integer-valued, the set
\[
\{n:b_n=B\}
\]
is either empty or a finite interval of consecutive indices. When this
set is nonempty, we refer to the corresponding interval as the
\emph{level-$B$ plateau}. If $b_n$ never takes the value $B$, then
\[
M(B)=M(B-1)+1.
\]
Otherwise, the level-$B$ plateau is
\[
n_1\leq n\leq n_2,
\qquad
n_1=N(B-1)+1,\quad n_2=N(B).
\]
Setting
\[
T=n_1+B=M(B-1),
\]
we have $b_i\leq B-1$ for $i<n_1$, and hence
\[
S=\sum_{i=1}^{n_1-1}a_i\leq T^2.
\]
If $n_2+B>2K_0T^{20}$ for a suitable absolute constant $K_0$, one can choose a power of two
$2^r$ such that
\[
K_0T^{20}<2^r\leq2K_0T^{20}<n_2+B.
\]
For such a choice, $n=2^r-B$ lies in the level-$B$ plateau and
$n\geq3$, so $a_n=2^r$ has a representation as a sum of at least two
consecutive earlier terms. This representation must cross the left
endpoint $n_1$ of the plateau: it cannot end before $n_1$, since
$2^r>S$, and it cannot begin at or after $n_1$, since a power of $2$
is not a sum of at least two consecutive positive integers. Splitting
the sum at $n_1$ and completing the square therefore gives
\[
2^m=x^2+D,
\qquad
D\neq0,\qquad |D|\leq13T^2.
\]
Beukers' estimate~\cite[Corollary~1]{Beukers1981} then gives $2^r<K_0T^{20}$, contradicting the choice of $r$. Consequently,
\[
M(B)\leq K M(B-1)^{20}
\]
for some absolute constant $K>0$. Iterating gives $\log M(B)\leq A20^B$ for some absolute constant $A>0$. Finally, taking $B=b_n$, we have $n\leq N(B)<M(B)$, and hence $\log n\leq A20^B$. Taking logarithms once more yields
\[
b_n\geq\frac{\log\log n}{\log 20}-O(1).
\]

The proof of Theorem~\ref{thm:upper-final} starts from the fact that,
from $a_3$ onward, the greedy sequence lists in increasing order all
integers representable as sums of at least two consecutive terms. Let
$\mathcal R_m$ denote the integers representable as sums of at least
two consecutive terms among $a_1,\ldots,a_m$. Recall that a finite ordered set
$Y=\{y_0<y_1<\cdots<y_m\}\subset\mathbb R$ is called \emph{convex} if
its consecutive gaps $y_i-y_{i-1}$ are strictly increasing. The
prefix-sum set $\widetilde S_m=\{s_0,s_1,\ldots,s_m\}$, where $s_0=0$
and $s_j=a_1+\cdots+a_j$, is convex in this sense because its
consecutive gaps are $a_1,\ldots,a_m$. Every positive difference of two
elements of $\widetilde S_m$ is either one of the $m$ one-term
differences $a_1,\ldots,a_m$ or belongs to $\mathcal R_m$. Bloom's lower bound for difference sets of convex sets~\cite[Theorem~2]{Blo25} therefore gives
\[
|\mathcal R_m|\gg_\eta m^\delta,
\qquad
\delta=\frac{6681}{4175}-\eta,
\]
for $0<\eta<2506/4175$. Feeding this count back into the greedy enumeration gives, for a suitable constant $C=C_\eta>0$ and all sufficiently large $n$, a recursion of the form
\[
a_n\leq Cn^{1/\delta}a_{\lceil Cn^{1/\delta}\rceil}.
\]
Iterating it gives, for some $B_\eta>0$,
\[
a_n\ll_\eta n^{1/(\delta-1)}(\log n)^{B_\eta}.
\]
Choosing $\eta$ sufficiently small in terms of $\varepsilon$ and absorbing the logarithmic factor into $n^\varepsilon$ gives the exponent $4175/2506+\varepsilon$.

\subsection{Notation}\label{subsec:notation}
We use Vinogradov's asymptotic notation. For functions $f=f(n)$ and $g=g(n)$, we write $f=O(g)$, $g=\Omega(f)$, $f\ll g$, or $g\gg f$ to mean that there exists a constant $C>0$ such that $|f(n)|\leq Cg(n)$ for all sufficiently large $n$. When we allow the implicit constant to depend on one or more parameters $z_1,\ldots,z_r$, we indicate this by writing $f=O_{z_1,\ldots,z_r}(g)$, $g=\Omega_{z_1,\ldots,z_r}(f)$, $f\ll_{z_1,\ldots,z_r} g$ or $g\gg_{z_1,\ldots,z_r} f$. We write $f\asymp g$ or $f=\Theta(g)$ to mean that $f\ll g$ and $g\ll f$, and write $f=o(g)$ to mean that $f(n)/g(n)\to0$ as $n\to\infty$. We write $f=\omega(g)$ to mean that $g=o(f)$; in particular, $\omega(1)$ denotes any quantity tending to infinity as $n\to\infty$.

\subsection{Paper organization}\label{subsec:paper-organization}
Section~\ref{sec:preliminaries} collects the two preliminary ingredients that we use later: the characterization of sums of consecutive positive integers and a Diophantine finiteness consequence of the theorem of Schinzel and Tijdeman. Section~\ref{sec:qualitative} proves Theorem~\ref{thm:qualitative}. Section~\ref{sec:quant-lower} establishes the quantitative lower bound in Theorem~\ref{thm:quantitative-lower}. Section~\ref{sec:upper-bound} proves the polynomial upper bound in Theorem~\ref{thm:upper-final}. Finally, Section~\ref{sec:conclude} contains concluding remarks on extensions, later improvements, and open questions.

\section{Preliminaries}\label{sec:preliminaries}

We start with a classical characterization of sums of consecutive
positive integers. Sylvester stated the enumeration theorem that the total number of
representations of a positive integer $N$ as a sum of one or more
consecutive positive integers equals the number of odd divisors of
$N$~\cite[\S~17, p.\ 265]{SylvesterFranklin1882}; the corresponding
bijection is established in the following section
\cite[\S~18, p.\ 266]{SylvesterFranklin1882}.
Equivalently, after excluding the trivial one-term representation, the
number of representations using at least two terms equals the number of
odd divisors of $N$ greater than $1$; see also Ho, He, and
Shiue~\cite[p.\ 242]{HoHeShiue2023}. Gamer, Roeder, and
Watkins call integers representable as sums of at least two consecutive
positive integers \emph{trapezoidal} and prove that the trapezoidal
positive integers are precisely those that are not powers of
$2$~\cite[Proposition~1, p.\ 108]{GamerRoederWatkins1985}. For the sake
of completeness, we include the short proof.

\begin{lemma}\label{lem:consecutive}
A positive integer $N$ admits a representation as a sum of at least two consecutive positive integers if and only if $N$ is not a power of $2$.
\end{lemma}

\begin{proof}
Suppose $N=x+(x+1)+\cdots+(x+\ell-1)$ with $\ell\geq2$ and $x\geq1$. Then
\[
2N=\ell(2x+\ell-1).
\]
If $\ell$ is odd, then $\ell\mid 2N$ implies $\ell\mid N$, so $N$ has an odd divisor $>1$.
If $\ell$ is even, then $2x+\ell-1$ is odd and divides $2N$. Hence $2x+\ell-1$ divides $N$, so $N$ again has
an odd divisor $>1$.  In either case $N$ cannot be a power of $2$.

Conversely, if $N$ is not a power of $2$, then $N$ has an odd divisor $d>1$, say $N=dm$.
If $m\geq (d+1)/2$, then
\[
N=\Bigl(m-\frac{d-1}{2}\Bigr)+\Bigl(m-\frac{d-1}{2}+1\Bigr)+\cdots+\Bigl(m+\frac{d-1}{2}\Bigr)
\]
is a sum of $d$ consecutive positive integers.  If $m<(d+1)/2$, then
\[
N=\Bigl(\frac{d+1}{2}-m\Bigr)+\Bigl(\frac{d+1}{2}-m+1\Bigr)+\cdots+\Bigl(\frac{d+1}{2}+m-1\Bigr)
\]
is a sum of $2m\geq2$ consecutive positive integers.
\end{proof}

Next we record the Diophantine ingredient. We shall use the following consequence of a result of Schinzel and Tijdeman~\cite[Corollary~1]{SchinzelTijdeman1976}. Here a complex zero $\xi$ of a polynomial $P$ is called \emph{simple} if $P'(\xi)\neq0$, equivalently, if it has multiplicity one. Thus ``at least two simple zeros'' means two distinct complex zeros, each of multiplicity one.
\begin{lemma}\label{lem:st}
Let $P(x)\in\mathbb{Q}[x]$ be a polynomial with at least two simple zeros. Then the equation
\[
y^k=P(x),\qquad x,y\in\ZZ,\ \ |y|>1,
\]
has only finitely many integer solutions $(x,y,k)$ with $k>2$.
\end{lemma}
\begin{corollary}\label{cor:quad}
Fix $E\in\ZZ$.  The equation
\[
v^2+v+E=2^k
\]
has only finitely many integer solutions $(v,k)\in\ZZ^2$ with $k>2$.
\end{corollary}
\begin{proof}
Apply Lemma~\ref{lem:st} with $y=2$ and $P(v)=v^2+v+E$. Since $E\in\ZZ$, the discriminant $1-4E$ is congruent to $1$ modulo $4$ and hence is nonzero. Therefore $P$ has two distinct zeros, each of multiplicity one, so the lemma applies.
\end{proof}

\section{A qualitative lower bound and infinitely many omissions}\label{sec:qualitative}

Throughout this section we write
\[
b_n:=a_n-n\qquad(n\geq1).
\]
Our goal is to prove Theorem~\ref{thm:qualitative}. The first step is a simple monotonicity observation.

\subsection{Monotonicity and eventual linearity}
\begin{lemma}\label{lem:monotone}
The sequence $(b_n)_{n\geq1}$ is nondecreasing.
\end{lemma}
\begin{proof}
For every $n\geq1$, the greedy construction gives $a_{n+1}>a_n$, and hence $a_{n+1}\geq a_n+1$. Therefore
\[
b_{n+1}-b_n=(a_{n+1}-(n+1))-(a_n-n)=a_{n+1}-a_n-1\geq0.
\qedhere\]
\end{proof}
\begin{lemma}\label{lem:linear}
If $(b_n)_{n\geq1}$ is bounded above, then there exist integers $n_0\geq2$ and $B$ such that
\[
a_n=n+B\qquad(n\geq n_0).
\]
\end{lemma}
\begin{proof}
By Lemma~\ref{lem:monotone}, the sequence $(b_n)_{n\geq1}$ is a nondecreasing sequence of integers. If it is bounded above, then it is eventually constant: there exist $n_0\geq2$ and an integer $B$ such that $b_n=B$ for all $n\geq n_0$. Hence $a_n=n+b_n=n+B$ for all $n\geq n_0$.
\end{proof}

We now assume, for contradiction, that $(b_n)_{n\geq1}$ is bounded above and hence that
\begin{equation}\label{eq:tail}
a_n=n+B\qquad(n\geq n_0)
\end{equation}
for some integers $n_0\geq 2$ and $B$. Since $b_1=0$ and $(b_n)$ is nondecreasing, we have $B\geq 0$.

\subsection{Powers of two on a linear plateau}

The following lemma contains the part common to the qualitative argument in this section and the quantitative argument in Section~\ref{sec:quant-lower}.
\begin{lemma}\label{lem:powers}
Let $B\geq 0$, and suppose that
\[
b_n=B
\]
for every integer $n$ with $n_1\leq n\leq n_2$, where $2\leq n_1\leq n_2$. Set
\[
T:=n_1+B,\qquad S:=\sum_{i=1}^{n_1-1}a_i.
\]
If $r\in\ZZ$ satisfies
\begin{equation}\label{eq:2rtsbs1}
\max\{T,S+1,B+3\}\leq 2^r\leq n_2+B,
\end{equation}
then there exist integers $C$ and $v\geq T$ such that
\begin{equation}\label{eq:quad}
2^{r+1}=v^2+v+2C-T(T-1),
\end{equation}
where $0\leq C\leq S$.
\end{lemma}
\begin{proof}
Set $n=2^r-B$. The lower bound $2^r\geq T$ and the upper bound in \eqref{eq:2rtsbs1} give $n_1\leq n\leq n_2$, so $a_n=n+B=2^r$. The bound $2^r\geq B+3$ gives $n\geq 3$. Hence there are indices $1\leq p<q\leq n-1$ such that
\[
2^r=\sum_{k=p}^{q}a_k.
\]
Since $2^r>S$, we must have $q\geq n_1$. If $p\geq n_1$, then
\[
2^r=\sum_{k=p}^{q}(k+B)=\sum_{j=p+B}^{q+B}j.
\]
Since $p\geq 1$ and $B\geq 0$, we have $p+B\geq 1$. Thus $2^r$ would be a sum of at least two consecutive positive integers, contradicting Lemma~\ref{lem:consecutive}. Therefore $p<n_1\leq q$. Put
\[
C:=\sum_{k=p}^{n_1-1}a_k,\qquad v:=q+B.
\]
Then $0\leq C\leq S$, $v\geq T$, and
\[
2^r=C+\sum_{j=T}^{v}j
=C+\frac{v(v+1)-T(T-1)}{2}.
\]
Multiplying by $2$ gives \eqref{eq:quad}. Here the condition $2^r\geq T$ gives $n\geq n_1$, the condition
$2^r\geq S+1$ gives $q\geq n_1$, the condition $2^r\geq B+3$
ensures that $n\geq3$ and hence that the defining consecutive-sum
representation is available, and the condition $2^r\leq n_2+B$
gives $n\leq n_2$.
\end{proof}

\subsection{Completion of the proof}
\begin{proof}[Proof of Theorem~\ref{thm:qualitative}]
Lemma~\ref{lem:monotone} shows that $(b_n)$ is nondecreasing. It remains to prove that it is unbounded.

Assume, for contradiction, that $(b_n)$ is bounded above. Then \eqref{eq:tail} holds by Lemma~\ref{lem:linear}. Set
\[
T:=n_0+B,\qquad S:=\sum_{i=1}^{n_0-1}a_i.
\]
For every sufficiently large integer $r$, we have
\[
2^r\geq\max\{T,S+1,B+3\}.
\]
Moreover, with $n_1=n_0$ and $n_2=2^r-B$, we have
\[
n_2\geq n_0
\qquad\text{and}\qquad
n_2+B=2^r.
\]
Thus Lemma~\ref{lem:powers} gives integers $C$ and $v\geq T$ such that
\[
2^{r+1}=v^2+v+E_C,\qquad E_C:=2C-T(T-1),
\]
where $0\leq C\leq S$. Since only finitely many integers $E_C$ can occur, the pigeonhole principle gives an integer $E$ for which
\[
v^2+v+E=2^k
\]
has infinitely many integer solutions $(v,k)\in\ZZ^2$ with $k>2$. This contradicts Corollary~\ref{cor:quad}. Therefore $(b_n)$ is unbounded.

Since $(b_n)$ is a nondecreasing unbounded sequence of integers, we have $b_n\to\infty$, and hence
\[
a_n=n+b_n=n+\omega(1).
\]
Finally, because $(a_n)$ is strictly increasing, exactly $n$ elements of $\{a_m:m\geq1\}$ lie in the interval $[1,a_n]$, namely $a_1,\ldots,a_n$. Thus the sequence omits exactly
\[
a_n-n=b_n
\]
positive integers in $[1,a_n]$. As $b_n\to\infty$, the set $\{a_n:n\geq1\}$ omits infinitely many positive integers.
\end{proof}

\section{A quantitative lower bound for the classical sequence}\label{sec:quant-lower}

We use the following explicit consequence of Beukers' work~\cite[Corollary~1]{Beukers1981}.
\begin{lemma}\label{lem:beukers-explicit}
Let $D\in\ZZ\setminus\{0\}$. Every integer solution $(x,m)\in\ZZ^2$ of
\[
x^2+D=2^m
\]
satisfies
\[
m<435+\frac{10\log|D|}{\log 2}.
\]
\end{lemma}
\begin{lemma}\label{lem:local-power-plateau}
Let $\widehat{B}\geq 1$, and suppose that
\[
b_n=\widehat{B}
\]
for every integer $n$ with $n_1\leq n\leq n_2$, where $2\leq n_1\leq n_2$. Set
\[
\widehat{T}:=n_1+\widehat{B},\qquad \widehat{S}:=\sum_{i=1}^{n_1-1}a_i.
\]
If $r\in\ZZ$ satisfies
\begin{equation}\label{eq:qqqwwweee1}
\max\{\widehat{T},\widehat{S}+1,\widehat{B}+3\}\leq 2^r\leq n_2+\widehat{B},
\end{equation}
then there exist integers $x$ and $D\neq 0$ such that
\[
2^{r+3}=x^2+D
\qquad\text{and}\qquad
|D|\leq 13\widehat{T}^2.
\]
\end{lemma}
\begin{proof}
Since $(b_n)_{n\geq 1}$ is nondecreasing by
Lemma~\ref{lem:monotone} and $b_{n_1}=\widehat{B}$, we have
\[
b_i\leq\widehat{B}\qquad(1\leq i<n_1).
\]
Hence, for every $1\leq i<n_1$,
\[
a_i=i+b_i\leq i+\widehat{B}\leq n_1+\widehat{B}
=\widehat{T}.
\]
Therefore
\[
\widehat{S}\leq(n_1-1)\widehat{T}\leq\widehat{T}^2.
\]
By \eqref{eq:qqqwwweee1}, Lemma~\ref{lem:powers}, applied with
$B=\widehat{B}$, gives integers $C$ and $v$ such that
$0\leq C\leq\widehat{S}$ and
\[
2^{r+1}=v^2+v+2C-\widehat{T}(\widehat{T}-1).
\]
Completing the square gives
\[
2^{r+3}=(2v+1)^2+D,
\qquad
D:=8C-4\widehat{T}(\widehat{T}-1)-1.
\]
Since $D\equiv-1\pmod 4$, we have $D\neq0$. Finally,
\[
|D|\leq 8C+4\widehat{T}(\widehat{T}-1)+1
\leq12\widehat{T}^2+1
\leq13\widehat{T}^2.
\]
The interval hypothesis on $b_n$ is used to invoke
Lemma~\ref{lem:powers}; monotonicity is used to bound the prefix sum
$\widehat{S}$; and the assumption $\widehat{B}\geq1$ is the range
needed in the recursion below.
\end{proof}

We are now ready to prove the quantitative lower bound for the sequence in Definition~\ref{def:hof}.
\begin{proof}[Proof of Theorem~\ref{thm:quantitative-lower}]
By Theorem~\ref{thm:qualitative}, the sequence $(b_n)$ is nondecreasing and unbounded. Hence, for each integer $\widehat{B}\geq 0$,
\[
N(\widehat{B}):=\max\{n\geq 1:\ b_n\leq \widehat{B}\}
\]
is well defined and finite. Set
\[
M(\widehat{B}):=N(\widehat{B})+\widehat{B}+2 \qquad (\widehat{B}\geq 0).
\]
We claim that there exists an absolute constant $K\geq 2$ such that
\begin{equation}\label{eq:MB-recursion}
M(\widehat{B})\leq K\,M(\widehat{B}-1)^{20}\qquad(\widehat{B}\geq 1).
\end{equation}

Fix $\widehat{B}\geq 1$. If $N(\widehat{B})=N(\widehat{B}-1)$, then
\[
M(\widehat{B})
=
N(\widehat{B})+\widehat{B}+2
=
N(\widehat{B}-1)+\widehat{B}+2
=
M(\widehat{B}-1)+1.
\]
Since $M(\widehat{B}-1)\geq 2$, it follows that
\[
M(\widehat{B})\leq M(\widehat{B}-1)^{20}.
\]

Suppose now that $N(\widehat{B})>N(\widehat{B}-1)$. Since $(b_n)$ is integer-valued
and nondecreasing, we then have
\[
b_n=\widehat{B}
\qquad
\bigl(N(\widehat{B}-1)+1\leq n\leq N(\widehat{B})\bigr).
\]
Write
\[
n_1:=N(\widehat{B}-1)+1,\qquad
n_2:=N(\widehat{B}),\qquad
\widehat{T}:=n_1+\widehat{B}.
\]
Then
\[
\widehat{T}
=
N(\widehat{B}-1)+1+\widehat{B}
=
N(\widehat{B}-1)+(\widehat{B}-1)+2
=
M(\widehat{B}-1).
\]

Let
\[
K_0:=2^{432}13^{10}.
\]
We shall show that
\[
n_2+\widehat{B}\leq 2K_0\widehat{T}^{20}.
\]
Assume to the contrary that
\[
n_2+\widehat{B}>2K_0\widehat{T}^{20}.
\]
Then there exists an integer $r$ such that
\[
K_0\widehat{T}^{20}<2^r\leq 2K_0\widehat{T}^{20}<n_2+\widehat{B}.
\]

We now verify the hypotheses of Lemma~\ref{lem:local-power-plateau}.
Since $\widehat{T}\geq 2$, we have
\[
2^r>K_0\widehat{T}^{20}\geq \widehat{T}^{20}>\widehat{T}^2.
\]
Moreover, for every $1\leq i\leq n_1-1$ we have $b_i\leq \widehat{B}-1$, because
$n_1=N(\widehat{B}-1)+1$. Hence
\[
a_i=i+b_i\leq i+\widehat{B}-1
\qquad (1\leq i\leq n_1-1),
\]
and therefore
\[
\sum_{i=1}^{n_1-1} a_i
\leq
\sum_{i=1}^{n_1-1}(i+\widehat{B}-1)
=
\frac{(n_1-1)(n_1+2\widehat{B}-2)}{2}.
\]
On the other hand,
\[
\widehat{T}^2=(n_1+\widehat{B})^2,
\]
and a direct calculation shows that
\[
2\widehat{T}^2-(n_1-1)(n_1+2\widehat{B}-2)
=
n_1^2+2n_1\widehat{B}+2\widehat{B}^2+3n_1+2\widehat{B}-2>0.
\]
Thus
\[
\sum_{i=1}^{n_1-1} a_i\leq \widehat{T}^2<2^r.
\]
Also, from $2^r>\widehat{T}^2$ and $\widehat{T}\geq 2$, we obtain
\[
2^r\geq \widehat{T}.
\]
Finally, since $n_1\geq 2$, we have
\[
\widehat{T}=n_1+\widehat{B}\geq \widehat{B}+2,
\]
and therefore
\[
2^r>\widehat{T}^2\geq (\widehat{B}+2)^2>\widehat{B}+3,
\]
so in particular
\[
2^r\geq \widehat{B}+3.
\]

Hence Lemma~\ref{lem:local-power-plateau} applies, and yields integers $x$ and
$D\neq 0$ such that
\[
2^{r+3}=x^2+D
\qquad\text{and}\qquad
|D|\leq 13\widehat{T}^2.
\]
Applying Lemma~\ref{lem:beukers-explicit}, we obtain
\[
r+3<435+\frac{10\log|D|}{\log 2}
\leq
435+\frac{10\log(13\widehat{T}^2)}{\log 2}.
\]
Exponentiating, we get
\[
2^r<2^{432}13^{10}\widehat{T}^{20}=K_0\widehat{T}^{20},
\]
contrary to the choice of $r$. This contradiction proves that
\[
n_2+\widehat{B}\leq 2K_0\widehat{T}^{20}.
\]

Since $n_2=N(\widehat{B})$ and $\widehat{T}=M(\widehat{B}-1)$, it follows that
\[
N(\widehat{B})+\widehat{B}\leq 2K_0\,M(\widehat{B}-1)^{20},
\]
and hence
\[
M(\widehat{B})
=
N(\widehat{B})+\widehat{B}+2
\leq
(2K_0+2)\,M(\widehat{B}-1)^{20}.
\]
Thus \eqref{eq:MB-recursion} holds with
\[
K:=2K_0+2.
\]

Taking logarithms in \eqref{eq:MB-recursion} and iterating, we obtain
\[
\log M(\widehat{B})
\leq
20^{\widehat{B}}\log M(0)
+
\bigl(1+20+\cdots+20^{\widehat{B}-1}\bigr)\log K
\]
for all $\widehat{B}\geq 0$. Thus
\[
\log M(\widehat{B})
\leq
20^{\widehat{B}}\log M(0)
+
\frac{20^{\widehat{B}}-1}{19}\log K
\leq
\widehat{A}\,20^{\widehat{B}},
\]
where
\[
\widehat{A}:=\log M(0)+\frac{\log K}{19}.
\]

Now fix $n\geq 2$ and put $\widehat{B}=b_n$. Then by definition of $N(\widehat{B})$, we have
$n\leq N(\widehat{B})<M(\widehat{B})$, and therefore
\[
\log n\leq \log M(\widehat{B})\leq \widehat{A}\,20^{\widehat{B}}.
\]
Taking logarithms once more, we get
\[
\widehat{B}=b_n\geq \frac{\log\log n-\log \widehat{A}}{\log 20}.
\]
Finally, since $a_n=n+b_n$, it follows that
\[
a_n\geq n+\frac{\log\log n}{\log 20}-\frac{\log \widehat{A}}{\log 20}.
\]
This proves the theorem.
\end{proof}

\section{A polynomial upper bound for the classical sequence}\label{sec:upper-bound}

In this section we prove Theorem~\ref{thm:upper-final}. The argument combines the greedy ordering of the representable integers with a lower bound for difference sets of finite convex sets.

\subsection{Greedy enumeration of the representable integers}

We begin by isolating the integers that admit representations as sums of at least two
consecutive terms. The key point is that the greedy construction does not merely produce
some such integers: it lists all of them in increasing order.

Define the partial sums
\[
s_0:=0,\qquad s_m:=\sum_{i=1}^m a_i\quad(m\geq 1),
\]
and the prefix-sum set
\[
\widetilde{S}_m:=\{s_0,s_1,\ldots,s_m\}.
\]
For $m\geq 2$ let
\[
\mathcal R_m:=\Bigl\{\sum_{i=p}^q a_i:\ 1\leq p<q\leq m\Bigr\}\subset \mathbb N,
\qquad
\mathcal R:=\bigcup_{m\geq 2}\mathcal R_m.
\]
Thus $\mathcal R$ is the set of all integers representable as a sum of at least two
consecutive terms of the sequence.
\begin{lemma}\label{lem:greedy-enumerates-R}
We have $\{a_k: k\geq 3\}=\mathcal R$, and $(a_3,a_4,\ldots)$ is the increasing enumeration
of $\mathcal R$. In particular, for every $X\geq 1$,
\begin{equation}\label{eq:count-implies-an}
\#(\mathcal R\cap[1,X])\ \geq\ n
\quad\Longrightarrow\quad
a_{n+2}\leq X.
\end{equation}
\end{lemma}

\begin{proof}
Every $a_k$ with $k\geq 3$ belongs to $\mathcal R$ by definition. Conversely, let $x\in\mathcal R$. Since $a_k\to\infty$, choose the least $k\geq 3$ such that $x\leq a_k$. Then $a_{k-1}<x\leq a_k$. Suppose that $x<a_k$, and write $x=\sum_{i=p}^{q}a_i$ with $p<q$. If $q\geq k$, then positivity of the summands and monotonicity of $(a_i)$ give $x\geq a_q\geq a_k$, a contradiction. Hence $q\leq k-1$. Since $a_{k-1}<x$ and $x$ is a sum of at least two consecutive terms among $a_1,\ldots,a_{k-1}$, the integer $x$ was one of the candidates in the definition of $a_k$. By the minimality of $a_k$, we must therefore have $a_k\leq x$. This contradicts the assumption $x<a_k$. Hence $x=a_k$. It follows that $(a_3,a_4,\ldots)$ is the increasing enumeration of $\mathcal R$, and \eqref{eq:count-implies-an} follows immediately.
\end{proof}

The preceding lemma reduces the problem of bounding $a_n$ to that of producing many
representable integers below a controlled threshold. The next observation makes this reduction
quantitative at a finite stage.
\begin{lemma}\label{lem:Rm-to-an}
For every $m\geq 2$, we have $\mathcal R_m\subseteq \mathcal R\cap[1,s_m]$. In particular, if
$|\mathcal R_m|\geq n$ then
\begin{equation}\label{eq:an-le-sm}
a_{n+2}\leq s_m.
\end{equation}
\end{lemma}

\begin{proof}
Every element of $\mathcal R_m$ is a sum of at least two consecutive terms among $a_1,\ldots,a_m$,
hence belongs to $\mathcal R$. Such a sum is clearly at most $a_1+\cdots+a_m=s_m$, so
$\mathcal R_m\subseteq \mathcal R\cap[1,s_m]$.

If $|\mathcal R_m|\geq n$, then $\#(\mathcal R\cap[1,s_m])\geq n$, and \eqref{eq:an-le-sm} follows from
Lemma~\ref{lem:greedy-enumerates-R} and \eqref{eq:count-implies-an}.
\end{proof}

We now turn to the task of estimating $|\mathcal R_m|$. The key idea is to rewrite sums of
consecutive terms as differences of prefix sums, thereby bringing in additive-combinatorial
results on convex sets.

\subsection{From consecutive sums to a convex difference set}

We recall the standard additive-combinatorial notion of convexity. We call a finite set \(Y=\{y_1<y_2<\cdots<y_n\}\subset \mathbb R\) \emph{convex} if its consecutive differences form a strictly increasing sequence, that is,
\[
y_{i+1}-y_i<y_{i+2}-y_{i+1}
\qquad(1\leq i\leq n-2).
\]
We now observe that the prefix-sum set $\widetilde{S}_m$ is convex in this sense.
\begin{lemma}\label{lem:prefix-sums-convex}
For each $m\geq 2$, the set $\widetilde{S}_m=\{s_0<s_1<\cdots<s_m\}$ is convex, i.e.
\[
s_i-s_{i-1}<s_{i+1}-s_i\qquad(1\leq i\leq m-1).
\]
\end{lemma}

\begin{proof}
We have $s_i-s_{i-1}=a_i$ for all $i\geq 1$, and $(a_i)$ is strictly increasing by construction.
\end{proof}

Since every sum of consecutive terms is a difference of two prefix sums, the difference set \(\widetilde{S}_m-\widetilde{S}_m\) provides a natural ambient set in which $\mathcal R_m$ sits. The next lemma makes this relation explicit.
\begin{lemma}\label{lem:Rm-from-diffset}
Let $D_m:=\widetilde{S}_m-\widetilde{S}_m=\{x-y:\ x,y\in \widetilde{S}_m\}$. Then
\[
|\mathcal R_m| \geq \frac{|D_m|-1}{2}-m.
\]
\end{lemma}

\begin{proof}
Let
\[
D_m^+:=\{d\in D_m:\ d>0\}.
\]
Since $D_m$ is symmetric and contains $0$, we have
\[
|D_m^+|=\frac{|D_m|-1}{2}.
\]
Every $d\in D_m^+$ has a representation of the form
\[
d=s_j-s_i=\sum_{t=i+1}^j a_t
\qquad(0\leq i<j\leq m).
\]
If $j-i=1$, then $d=a_j\in\{a_1,\ldots,a_m\}$.
If $j-i\geq2$, then $d\in\mathcal R_m$.
Therefore
\[
D_m^+\subseteq \{a_1,\ldots,a_m\}\cup \mathcal R_m.
\]
Hence
\[
\frac{|D_m|-1}{2}=|D_m^+|\leq m+|\mathcal R_m|,
\]
which rearranges to
\[
|\mathcal R_m|\geq \frac{|D_m|-1}{2}-m.
\qedhere\]
\end{proof}

\subsection{Bloom's convex difference-set bound and the basic recursion}

We now combine the convexity of the prefix-sum set with a recent difference-set estimate of Bloom. This gives a lower bound for $|\mathcal R_m|$ and, via the greedy enumeration above, a recursive upper bound for $a_n$.

Elekes, Nathanson, and Ruzsa~\cite{ElekesNathansonRuzsa2000} established foundational additive-growth estimates for convex sets. Schoen and Shkredov~\cite{SchoenShkredov2011} later proved the predecessor bound
\[
|A-A|\gg |A|^{8/5}(\log |A|)^{-2/5}.
\]
We use the following more recent difference-set bound of Bloom~\cite[Theorem~2]{Blo25}.
\begin{theorem}\label{thm:Bloom-convex-diff}
For every $\eta>0$ there exists $c_\eta>0$ such that for every finite convex set $A\subset\mathbb R$,
\[
|A-A| \geq c_\eta\,|A|^{6681/4175-\eta}.
\]
\end{theorem}

We apply this estimate to the convex set $\widetilde{S}_m$ and then use Lemma~\ref{lem:Rm-from-diffset}
to obtain the following lower bound for the number of representable integers up to stage $m$.
\begin{proposition}\label{prop:Rm-lower-bloom}
Fix $0<\eta<2506/4175$ and set $\delta=6681/4175-\eta$. Then there exist constants $m_0\in\mathbb N$ and $c>0$
(depending only on $\eta$) such that for all $m\geq m_0$,
\[
|\mathcal R_m| \geq c\, m^{\delta}.
\]
\end{proposition}

\begin{proof}
By Lemma~\ref{lem:prefix-sums-convex}, the set $\widetilde{S}_m$ is convex and $|\widetilde{S}_m|=m+1$. Applying
Theorem~\ref{thm:Bloom-convex-diff} to $A=\widetilde{S}_m$ gives
\[
|D_m|=|\widetilde{S}_m-\widetilde{S}_m| \geq c_\eta\,(m+1)^{\delta}.
\]
Insert this into Lemma~\ref{lem:Rm-from-diffset}:
\[
|\mathcal R_m| \geq \frac{c_\eta(m+1)^{\delta}-1}{2}-m.
\]
For all sufficiently large $m$ the right-hand side is $\geq (c_\eta/4)\,m^{\delta}$, which proves this
proposition with suitable choices of $m_0$ and $c$.
\end{proof}

The proposition shows that $\mathcal R_m$ grows polynomially with exponent $\delta>1$. We now feed this
counting information back into Lemma~\ref{lem:Rm-to-an} to obtain the basic recursion for $(a_n)$.
\begin{lemma}\label{lem:basic-recursion}
Fix $0<\eta<2506/4175$ and let $\delta=6681/4175-\eta>1$. Then there exist constants $C_1\geq 2$
and $N_0$ (depending only on $\eta$) such that for all $n\geq N_0$,
\begin{equation}\label{eq:recursion}
a_n\ \leq\ C_1\, n^{1/\delta}\, a_{\lceil C_1\,n^{1/\delta}\rceil}.
\end{equation}
\end{lemma}

\begin{proof}
Let $c,m_0$ be as in Proposition~\ref{prop:Rm-lower-bloom}. For $N$ large, set
\[
m_1:=\left\lceil (N/c)^{1/\delta}\right\rceil.
\]
Then $m_1\geq m_0$ and $c m_1^\delta \geq N$, so $|\mathcal R_{m_1}|\geq N$. By Lemma~\ref{lem:Rm-to-an},
\[
a_{N+2}\leq s_{m_1}.
\]
Since $a_i\leq a_{m_1}$ for $1\leq i\leq m_1$, we have $s_{m_1}\leq m_1 a_{m_1}$, and hence
\[
a_{N+2}\leq m_1 a_{m_1}.
\]
Moreover, for $N$ large enough,
\[
m_1=\left\lceil (N/c)^{1/\delta}\right\rceil \leq 2c^{-1/\delta}N^{1/\delta}.
\]
Set
\[
C_1:=\max\{2,\,2c^{-1/\delta}\}.
\]
Using the monotonicity of $(a_n)$, we obtain
\[
a_{N+2}\leq C_1 N^{1/\delta} a_{\lceil C_1 N^{1/\delta}\rceil}
\]
for all sufficiently large $N$.

Now replace $N$ by $n-2$. For all sufficiently large $n$,
\[
a_n \leq C_1 (n-2)^{1/\delta} a_{\lceil C_1 (n-2)^{1/\delta}\rceil}
\leq C_1 n^{1/\delta} a_{\lceil C_1 n^{1/\delta}\rceil}.
\]
This is \eqref{eq:recursion} after renaming the lower threshold as $N_0$.
\end{proof}

\subsection{Iteration and the final exponent}
We now iterate the basic recursion from Lemma~\ref{lem:basic-recursion}. The next lemma shows that any nondecreasing sequence satisfying such a recursion must obey a polynomial upper bound, up to a logarithmic factor.
\begin{lemma}\label{lem:iterate-recursion}
Let $(a_n)$ be nondecreasing and suppose there exist constants $C_1\geq 2$, $\alpha\in(0,1)$ and
$N_0$ such that
\[
a_n \leq C_1\, n^\alpha\, a_{\lceil C_1 n^\alpha\rceil}
\qquad(n\geq N_0).
\]
Then there exist constants $B,C_2>0$ such that for all $n\geq 3$,
\[
a_n \leq C_2\, n^{\alpha/(1-\alpha)}\,(\log n)^{B}.
\]
\end{lemma}

\begin{proof}
Enlarge $N_0$ if necessary so that
\[
\lceil C_1 x^\alpha\rceil \leq 2C_1 x^\alpha \leq x-1
\qquad(x\geq N_0),
\]
and also
\[
\frac{\log(2C_1)}{1-\alpha}\leq \frac12\log N_0.
\]

Fix $n> N_0$, and define
\[
\widetilde{n}_0:=n,\qquad \widetilde{n}_{j+1}:=\lceil C_1 \widetilde{n}_j^\alpha\rceil \quad (j\geq 0)
\]
for as long as $\widetilde{n}_j>N_0$. Since $\widetilde{n}_{j+1}\leq \widetilde{n}_j-1$ whenever $\widetilde{n}_j\geq N_0$, the positive integer sequence $(\widetilde{n}_j)$
strictly decreases until it first reaches the interval $[1,N_0]$. Hence there exists
a minimal $t=t(n)$ such that $\widetilde{n}_t\leq N_0$.

Repeatedly applying the assumed recursion gives
\[
a_n\ \leq\ \Bigl(\prod_{j=0}^{t-1} C_1 \widetilde{n}_j^\alpha\Bigr)\, a_{\widetilde{n}_t}.
\]
Since $\widetilde{n}_t\leq N_0$, we have
\[
a_{\widetilde{n}_t}\leq C_0:=\max_{1\leq r\leq N_0} a_r.
\]
Also, from \(\widetilde{n}_{j+1}\leq 2C_1 \widetilde{n}_j^\alpha\) we obtain by induction
\[
\widetilde{n}_j\ \leq\ (2C_1)^{1+\alpha+\cdots+\alpha^{j-1}}\, n^{\alpha^j}
\ \leq\ (2C_1)^{1/(1-\alpha)}\, n^{\alpha^j}.
\]
Therefore
\[
\prod_{j=0}^{t-1} \widetilde{n}_j^\alpha
\leq (2C_1)^{\alpha t/(1-\alpha)}\,
n^{\alpha+\alpha^2+\cdots+\alpha^t}
\leq (2C_1)^{\alpha t/(1-\alpha)}\, n^{\alpha/(1-\alpha)}.
\]
Combining the previous estimates yields
\[
a_n\ \leq\ C_0\, C_1^{\,t}\,(2C_1)^{\alpha t/(1-\alpha)}\, n^{\alpha/(1-\alpha)}.
\]

It remains to bound $t$. Write \(x_j:=\log \widetilde{n}_j\). Since
\[
\widetilde{n}_{j+1}\leq 2C_1 \widetilde{n}_j^\alpha,
\]
we have
\[
x_{j+1}\leq \alpha x_j+\log(2C_1).
\]
Iterating this inequality gives
\[
x_j\leq \alpha^j\log n+\frac{\log(2C_1)}{1-\alpha}.
\]
By the choice of $N_0$,
\[
\frac{\log(2C_1)}{1-\alpha}\leq \frac12\log N_0.
\]
Hence, whenever
\[
\alpha^j\log n\leq \frac12\log N_0,
\]
we obtain \(x_j\leq \log N_0\), and therefore $\widetilde{n}_j\leq N_0$. Thus $t$ is at most the least integer $j$ such that
\[
\alpha^j\log n\leq \frac12\log N_0.
\]
Equivalently, since $0<\alpha<1$,
\[
j\geq \frac{\log\log n-\log\bigl(\frac12\log N_0\bigr)}{\log(1/\alpha)}.
\]
It follows that
\[
t\ll \log\log n.
\]
Consequently,
\[
C_1^t(2C_1)^{\alpha t/(1-\alpha)}=(\log n)^{O(1)},
\]
and the claimed bound follows for all $n> N_0$. Enlarging $C_2$ covers the finitely many
values $3\leq n\leq N_0$.
\end{proof}

We are now ready to present the following.

\begin{proof}[Proof of Theorem~\ref{thm:upper-final}]
Fix $\varepsilon>0$. Choose $0<\eta<2506/4175$ so small that
\begin{equation}\label{eq:choose-eta}
\frac{1}{(6681/4175-\eta)-1}\ \leq\ \frac{4175}{2506}+\frac{\varepsilon}{2}.
\end{equation}
Set $\delta=6681/4175-\eta$ and $\alpha=1/\delta\in(0,1)$. By Lemma~\ref{lem:basic-recursion}, there exist constants $C_1\geq 2$ and $N_0$ such that
\[
a_n \leq C_1\, n^{\alpha}\, a_{\lceil C_1 n^{\alpha}\rceil}
\qquad(n\geq N_0).
\]
Applying Lemma~\ref{lem:iterate-recursion} yields constants $B,C_2>0$ such that
\[
a_n \leq C_2\, n^{\alpha/(1-\alpha)}\,(\log n)^{B}
\qquad(n\geq 3).
\]
Since $\alpha/(1-\alpha)=1/(\delta-1)$, the choice \eqref{eq:choose-eta} gives
\[
n^{\alpha/(1-\alpha)}\ \leq\ n^{4175/2506+\varepsilon/2}.
\]
Finally, $(\log n)^B\leq n^{\varepsilon/2}$ for all sufficiently large $n$, so
\[
a_n\ \leq\ C_\varepsilon\, n^{4175/2506+\varepsilon}
\]
for all large $n$, and we adjust $C_\varepsilon$ to cover the finitely many smaller $n$.
\end{proof}

\section{Concluding remarks}\label{sec:conclude}

\subsection{Generalized seeds of length at least two}

Although the body of this paper has focused exclusively on the classical seed $(1,2)$, one may more generally start with any fixed finite sequence of positive integers $u_1,\ldots,u_s$, where $s\geq2$, and thereafter define each new term by the same greedy rule. The preceding arguments appear to extend to this setting after suitable modifications involving the finite initial segment. If the initial seed is not increasing, the analogue of $u_n-n$ is nondecreasing only from the beginning of the strictly increasing greedy tail, and the convex prefix-sum set used in Section~\ref{sec:upper-bound} should be formed from that tail rather than from the entire sequence. Additional bookkeeping is also needed for consecutive sums that meet the initial seed. We do not pursue these modifications here and make no formal claim concerning generalized seeds.

\subsection{Later upper-bound improvements}

After we completed this paper, Sothanaphan pointed out in the comments on Erd\H{o}s Problem \#423~\cite{EP} that combining Cushman's recent bound
\[
|A-A|\gg_{\varepsilon}|A|^{8/5+1/3440-\varepsilon}
\]
for finite convex sets with the argument of Section~\ref{sec:upper-bound} yields the improved estimate
\[
a_n\ll_{\varepsilon} n^{688/413+\varepsilon}.
\]
See~\cite{Sothanaphan423,Cushman2025}. Since this refinement does not alter the underlying method, we have chosen to keep the original exposition.

\subsection{Open questions}

A substantial gap remains between Theorems~\ref{thm:quantitative-lower}
and~\ref{thm:upper-final}. It seems likely that the true growth of \(a_n\) is
much closer to linear than our present upper bound indicates. Writing
\[
b_n:=a_n-n,
\]
a natural next step is to determine the true order of magnitude of \(b_n\).
Most notably, one may ask whether \(b_n\) is sublinear.
\begin{problem}
Is it true that
\[
a_n=n+o(n)?
\]
Equivalently, does \(b_n=o(n)\) hold?
\end{problem}

To complement this problem, we computed $b_n$ for $1\leq n\leq30000$ and plotted the normalized quantities $b_n/n^{1/k}$ for $k=5,4,3,2$; see Figure~\ref{fig:bn-ratios}. The ratio $b_n/n^{1/2}$ appears to decrease with $n$, whereas $b_n/n^{1/5}$ appears to increase, which is consistent with a power-law heuristic of the form
\[
b_n=\Theta\bigl(n^{\alpha+o(1)}\bigr)
\qquad(n\to\infty)
\]
for some exponent $\alpha$ satisfying $1/5<\alpha<1/2$. Among these plots, the curve $b_n/n^{1/3}$ appears comparatively flat on the computed range, suggesting that $\alpha$ might be close to $1/3$, although the available data are far from decisive.

\begin{figure}[htbp]
  \centering
  \includegraphics[width=0.92\linewidth]{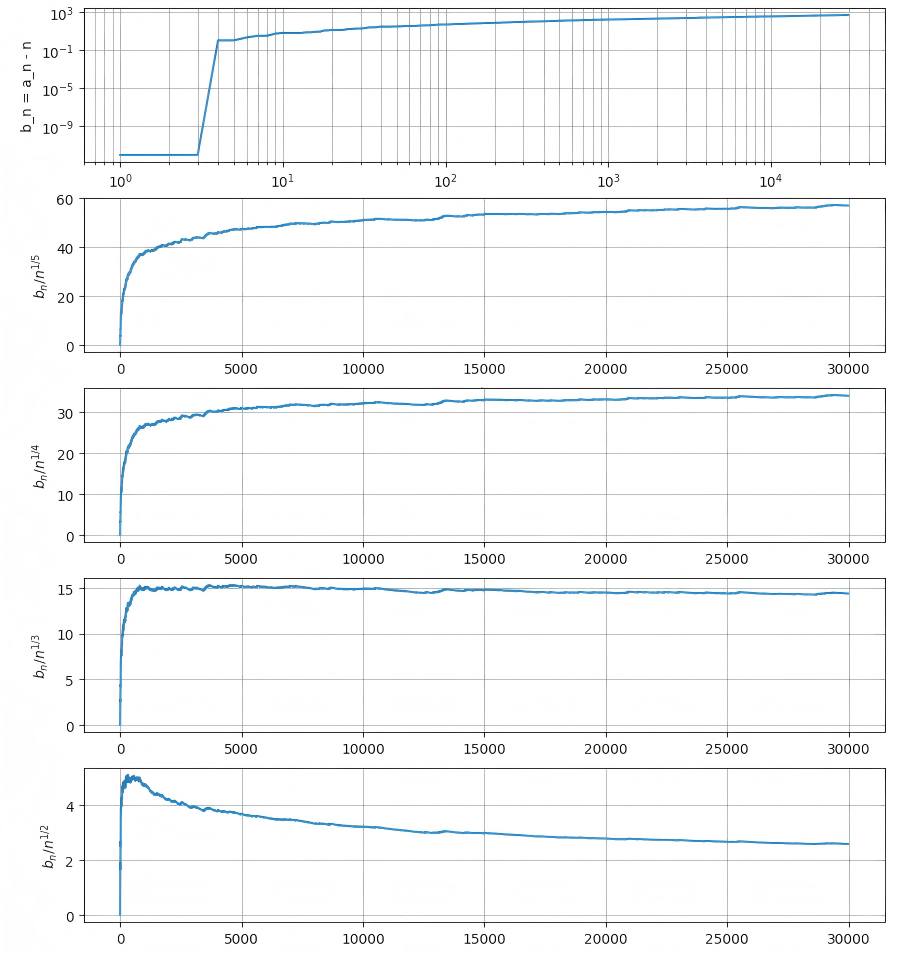}
  \caption{Numerical data for the sequence in Definition~\ref{def:hof}: the top panel shows $b_n=a_n-n$ (on a logarithmic $n$-axis), and the remaining panels plot the ratios $b_n/n^{1/k}$ for $k=5,4,3,2$ over $1\leq n\leq30000$.}
  \label{fig:bn-ratios}
\end{figure}

\section{Acknowledgments}\label{sec:acknowledgments}
An earlier version of this paper, containing only the qualitative conclusion $a_n=n+\omega(1)$, appeared as a short note in the comments section of the \emph{Erd\H{o}s Problems} website~\cite{EP}. Matthew Bolan also independently obtained the qualitative result that the classical sequence omits infinitely many positive integers; see again the comments section of the \emph{Erd\H{o}s Problems} website~\cite{EP}. We are grateful to Ingo Alth\"ofer for encouraging us to think about extensions to more general initial seeds. We thank Yann Bugeaud for pointing out that one can strengthen the method in our proof of $a_n=n+\omega(1)$ to obtain the bound of Theorem~\ref{thm:quantitative-lower}. We thank Norbert Hegyv\'ari for bringing the papers~\cite{Beker2024,Hegyvari1986,Tao2026} on consecutive sums to our attention, and Boris Alexeev and Wolfram Bernhardt for computing and sharing the numerical data posted in the comments section of the \emph{Erd\H{o}s Problems} website~\cite{EP}. Finally, we thank the editor for careful editorial guidance and the anonymous referees for their helpful comments and suggestions, which improved the presentation of the manuscript.

\medskip

{\noindent\bf Declaration of AI usage.}
The author used an AI assistant only during the brainstorming phase related to Theorem~\ref{thm:upper-final} and for drafting the Python code that generated Figure~\ref{fig:bn-ratios}. Apart from this limited assistance, the author carried out all mathematical arguments, proofs, and writing.

\bigskip
\hrule
\bigskip

\noindent 2020 {\it Mathematics Subject Classification}:
Primary 11B83; Secondary 11D61.

\noindent \emph{Keywords: } self-generating sequence, consecutive-term sum

\bigskip
\hrule
\bigskip

\noindent (Concerned with sequence
\seqnum{A005243}.)

\bigskip
\hrule
\bigskip

\vspace*{+.1in}
\noindent
Received
revised version received
Published in {\it Journal of Integer Sequences},

\bigskip
\hrule
\bigskip

\noindent
Return to \href{https://cs.uwaterloo.ca/journals/JIS/}{Journal of Integer Sequences home page}.
\vskip .1in

\end{document}